\documentclass{amsart}
\usepackage{amsmath, amssymb,epic,graphicx,mathrsfs,enumerate}
\usepackage[all]{xy}
\usepackage{color}

\usepackage{amsthm}
\usepackage{amssymb}
\usepackage{latexsym}
\usepackage{longtable}
\usepackage{epsfig}
\usepackage{amsmath}
\usepackage{hhline}

 \DeclareMathOperator{\perm}{Sym}
 
\DeclareMathOperator{\aut}{Aut}

\DeclareMathOperator{\F}{F}

 \DeclareMathOperator{\frat}{Frat}

  \DeclareMathOperator{\diam}{diam}

\DeclareMathOperator{\gl}{GL} 
 
\DeclareMathOperator{\GL}{GL}

\newcommand{\g}{\gamma}

\DeclareMathOperator{\z}{{\mathbb{Z}}}
\newtheorem{thm}{Theorem}
\newtheorem{cor}[thm]{Corollary}
 \newtheorem{lemma}[thm]{Lemma}

\numberwithin{equation}{section}

\renewcommand{\footnote}{\endnote}
\newcommand{\ignore}[1]{}\makeglossary

\begin{document}
	\bibliographystyle{amsplain}
	\title[The generating graph of the abelian groups]{The generating graph of the abelian groups}
\author{Cristina Acciarri}
\address{Cristina Acciarri\\ Department of Mathematics, University of Brasilia, 70910-900 Bras\'ilia DF, Brazil\\ 
email:acciarricristina@yahoo.it}
	\author{Andrea Lucchini}
\address{Andrea Lucchini\\ Universit\`a degli Studi di Padova\\  Dipartimento di Matematica \lq\lq Tullio Levi-Civita\rq\rq\\ Via Trieste 63, 35121 Padova, Italy\\email: lucchini@math.unipd.it}
\thanks{Research partially supported by MIUR-Italy via PRIN \lq Group theory and applications\rq. The first author is also supported by  Conselho Nacional de Desenvolvimento Cient\'ifico e Tecnol\'ogico CNPq-Brazil and FAPDF}	
		
	\begin{abstract} For a group $G,$ let $\Gamma(G)$ denote the graph defined on the 
		elements of $G$ in such a way that two distinct vertices are connected by an
		edge if and only if they generate $G$. Moreover let $\Gamma^*(G)$ be the subgraph of $\Gamma(G)$ that is induced by all  the vertices of $\Gamma(G)$ that are not isolated. We prove that if $G$ is a 2-generated non-cyclic abelian group then $\Gamma^*(G)$ is connected. Moreover $\diam(\Gamma^*(G))=2$ if the torsion subgroup of $G$ is non-trivial and $\diam(\Gamma^*(G))=\infty$ otherwise. If $F$ is the free group of rank 2, then  $\Gamma^*(F)$ is connected and we deduce from $\diam(\Gamma^*(\z\times \z))=\infty$ that $\diam(\Gamma^*(F))=\infty.$

			\end{abstract}
	\maketitle
For any group $G,$ the generating graph associated to $G,$ written  $\Gamma(G),$ is the graph where the vertices are the  elements of $G$ and we draw an edge between $g_1$ and $g_2$ if $G$ is generated by $g_1$ and $g_2.$ If $G$ is not 2-generated, then there will be no edge in this graph. Thus, it is natural to assume that $G$ is 2-generated. 	

There could be many isolated vertices in this graph. All of the elements in the Frattini subgroup will be isolated vertices, but we can also find isolated vertices outside the Frattini subgroup (for example the nontrivial elements of the Klein subgroup are isolated vertices in  $\Gamma(\perm(4)).$ 

Let $\Gamma^*(G)$ be the subgraph of $\Gamma(G)$ that is induced by all  the vertices that are not isolated.
In \cite{CLis} it is  proved that if $G$ is a 2-generated finite soluble group, then $\Gamma^*(G)$ is connected. Later on in \cite{diam} the second author investigated the diameter of such a graph  proving that  if $G$ is a finite soluble group then $\diam(\Gamma^*(G))\leq 3$. This bound is best possible, however 
$\diam(\Gamma^*(G))\leq 2$ in some relevant cases. In particular $\diam(\Gamma^*(G))\leq 2$ whenever $G$ is a finite abelian group.

In this paper we want to analyse the graph $\Gamma^*(G)$ when $G$ is an infinite 2-generated abelian group. If $G$ is an infinite cyclic group then the graph $\Gamma(G)$ coincides with $\Gamma^*(G)$ and it is connected with diameter 2 (indeed any vertex of $\Gamma(G)$ is adjacent to a vertex corresponding to a generator of $G$). So we may assume that $G$ is non-cyclic. There are two different cases. First we will consider the situation when the torsion subgroup $T(G)$ of $G$ is non-trivial. This means that $G= \z\times\z/n\z$ for some integer $n\geq 2$. In this case we will see that an argument relying on the Dirichlet's theorem on arithmetic progressions allows us to 
prove that $\Gamma^*(G)$ is connected and $\diam(\Gamma^*(G))=2.$ The situation is more intriguing when $T(G)$ is trivial, i.e. when $G\cong \z \times \z.$ In this latter case we will deduce that $\Gamma^*(G)$ is a connected graph using the fact that every $2\times 2$ invertible matrix over $\mathbb Z$ is a product of elementary matrices (\cite[Theorem 14.2]{om}).  However, although every element $A\in \GL(2,\z)$ can be written as a product of elementary matrices,
the number of elementary matrices required depends on the entries of $A$ and it is
not bounded. As explained for example in \cite{ck}, the reason for this is related to the fact that arbitrarily many divisions with remainder may be required when any (general) Euclidean algorithm is used to find the greatest common divisor of two integers. This leads to conjecture that the diameter of $\Gamma^*(\z\times \z)$ could be infinite. We will confirm this conjecture proving that the fact that division chains in $\z$ can be arbitrarily long (as happens for example when two consecutive Fibonacci numbers are considered) implies that the distance between two vertices of $\Gamma^*(\z\times \z)$ can be also arbitrarily large. Summarizing we have the following result.
\begin{thm}
Let $G$ be a 2-generated abelian group. Then $\Gamma^*(G)$ is connected. Moreover if
either  $G$ is cyclic or the torsion subgroup of $G$ is non-trivial, then $\diam(\Gamma^*(G))=2,$ while $\diam(\Gamma^*(\z\times \z))=\infty.$
\end{thm}
Finally let $F$ be the free group of rank 2. The fact that every automorphism of $F$ can be expressed as a product of elementary Nielsen transformations can be used to prove that the graph
$\Gamma^*(F)$ is connected (see Theorem \ref{confree}).  Moreover we will deduce from $\diam(\Gamma^*(\z\times \z)=\infty$ that also the graph  $\Gamma^*(F)$ cannot have a finite diameter (see Corollary \ref{fine}).

\section{Abelian groups with a non-trivial torsion subgroup}

If $G$ is a 2-generated finite abelian group, then $\Gamma^*(G)$ is
connected with diameter at most 2. This can be deduced from \cite[Corollary 3]{diam}, however a short alternative easy proof can be also given.  Indeed, we can write $G=P_1\times \dots\times P_t$ as a direct product of its Sylow subgroups. It can be easily seen that $\Gamma^*(G)$ is the cross product of the graphs  $\Gamma^*(P_1),\dots,\Gamma^*(P_t)$  so it suffices to prove that the graph
$\Gamma^*(G)$ is connected with diameter at most 2 in the particular case when $G$ is a 2-generated abelian $p$-group. On the other hand $\Gamma^*(G)=\Gamma^*(G/\frat(G))$
so we may assume that $G$ is an elementary abelian $p$-group of rank at most 2, and the conclusion follows from the fact that given two non-zero vectors $v_1,$ $v_2$ of a 2-dimensional vector space, there always exists a third vector $w$ such that $\langle v_1,w\rangle=\langle v_2,w\rangle.$

In the remaining part of this section, we will assume that $G$ is an infinite 2-generated abelian group and that the torsion subgroup of $G$ is non-trivial.
This means that there exists $n\geq 2$ such that $G= \z\times\z/n\z$.  In what follows,  we will  denote by $\bar{x}$ the image of  an integer $x$ in $\z/n\z$ under the natural epimorphism.

\begin{lemma}\label{generation} Let $G= \z\times\z/n\z$ and $a,b,c,d \in \z$.  The vertices $(a,\bar{b})$ and $(c,\bar{d})$ are adjacent in $\Gamma(G)$ if and only if the following two conditions hold:

\begin{itemize}
\item[(i)] $\mathrm{gcd}(a,c)=1$;
\item[(ii)] $\det\begin{pmatrix}a&b\\c&d\end{pmatrix}$ is coprime to $n$.
\end{itemize}

\end{lemma}

\begin{proof}
	We have that $(a,\bar{b})$ and $(c,\bar{d})$ generate $G$ if and only if $(a,b), (c,d)$ and $(0,n)$ generate $\mathbb Z \times \mathbb Z.$
The latter condition is equivalent to say that the invariant factors  of the matrix $$A=\begin{pmatrix}a&b\\
c&d\\0&n
\end{pmatrix}$$ are invertible elements of $\mathbb Z$, i.e.
\begin{enumerate}\item $\mathrm{gcd}(a,b,c,d,0,n)=1$;
	\item $\mathrm{gcd}\left(\det\begin{pmatrix}a&b\\c&d\end{pmatrix},
\det\begin{pmatrix}a&b\\0&n\end{pmatrix},\det\begin{pmatrix}c&d \\0&n\end{pmatrix}\right)=1.$
\end{enumerate}
Notice that $\mathrm{gcd}(a,b)$ divides $\mathrm{gcd}(ad-bc,an,bn),$ so condition (2) implies $\mathrm{gcd}(a,b)=1.$ Once we know that $\mathrm{gcd}(a,b)=1,$ we have that
$\mathrm{gcd}(an,bn)=n$, and consequently
$\mathrm{gcd}\left(\det\begin{pmatrix}a&b\\c&d\end{pmatrix},
\det\begin{pmatrix}a&b\\0&n\end{pmatrix},\det\begin{pmatrix}c&d\\0&n\end{pmatrix}\right)=\mathrm{gcd}\left(\det\begin{pmatrix}a&b\\c&d\end{pmatrix},n\right).$
\end{proof}

\begin{thm}\label{abeliantorsion}
Let $G= \z\times\z/n\z$ for some fixed $n\geq 2$.
Then $\Gamma^*(G)$ is connected and $\diam(\Gamma^*(G))=2.$
\end{thm}

\begin{proof}  Let $(a,\bar{b})$ and $(c,\bar{d})$ be   two non-isolated vertices  in  $\Gamma(G)$.  Consider $H= \z/n\z\times\z/n\z$.  As observed above $\diam(\Gamma^*(H))=2,$  hence taking the corresponding  non-isolated vertices  $(\bar{a},\bar{b})$ and $(\bar{c},\bar{d})$ in $\Gamma(H)$, there exists a pair $(x,y)$ of positive integers such that 
\[
	\det\begin{pmatrix}a&b\\x&y\end{pmatrix} \quad \text { and } \quad \det\begin{pmatrix}c&d\\x&y\end{pmatrix}
	\]
are both coprime to $n$.  Next we will show that there exists $t\in \z$ such that   $(x+tn,\bar{y})$ is adjacent simultaneously to $(a,\bar{b})$  and $(c,\bar{d})$ in $\Gamma(G)$.
By Lemma \ref{generation} it is enough to find $t\in \z$ satisfying the following conditions:

\begin{equation}\label{eq1}  \mathrm{gcd}(a,x+tn)=1=\mathrm{gcd}(c,x+tn).  \end{equation}

Indeed, by construction,  both  determinants
\[
	\det\begin{pmatrix}a&b\\x+tn&y\end{pmatrix} \quad \text { and } \quad \det\begin{pmatrix}c&d\\x+tn&y\end{pmatrix}
	\]
are  coprime to $n$.
Let $u=\mathrm{gcd}(x,n)$. Notice that  $\mathrm{gcd}(a,u)$ divides both $n$ and $\det\begin{pmatrix}a&b\\x&y\end{pmatrix}$, so
$\mathrm{gcd}(a,u)=1.$ Similarly $\mathrm{gcd}(b,u)=1.$	
Moreover there exist positive integers $x^*$ and $m$  such that $x=ux^*$ and $n=um$.  Observe that $\mathrm{gcd}(x^*,m)=1$ and so, by Dirichlet's theorem on arithmetic progressions, the set $\{x^*+tm \mid t\in \z\}$ contains infinitely many primes.  Take now $t\in \z$ such that $x^*+tm=p$ is a prime greater than $a$ and $c$. We may observe that $up=d(x^*+tm)=x+tn$ is the desired element satisfying both conditions in (\ref{eq1}).
\end{proof}

\section{The abelian free group of rank 2}
In this section we analyse the case where $G$ is a 2-generated non-cyclic torsion-free abelian  group, that is $G=\z\times \z$. First of all note that  $(a,b)$ and $(x,y)$ is a generating pair for $G$ if and only if $\det\begin{pmatrix}a&b\\x&y\end{pmatrix}=\pm 1$, so the  edges of $\Gamma^*(G)$ are in correspondence with the elements of $\gl(2,\z)$, up to a swap obtained by multiplying the matrix  $J=\begin{pmatrix}0&1\\1&0\end{pmatrix}$ on the left. Indeed, since the graph $\Gamma^*(G)$ is undirected, any generating pair  $(a,b)$ and $(x,y)$ of $G$ can be  represented equivalently   by either  the matrix $A=\begin{pmatrix}a&b\\x&y\end{pmatrix}$ or the matrix  $A^*=\begin{pmatrix}x&y\\a&b\end{pmatrix}=JA$. 

\begin{lemma}\label{nonisol} Let $(a,b)$ be a vertex of $\Gamma(G)$. Then $(a,b)$ is non-isolated if and only if $a$ and $b$ are coprime. 
\end{lemma} 

\begin{proof}
A pair $(a,b)$ is a non-isolated vertex of $\Gamma(G)$ if and only if   there exists a pair of integers $(x,y)$ such that  $\det\begin{pmatrix}a&b\\x&y\end{pmatrix}=\pm 1$. Namely, there exists an integer solution $(x,y)$ either for the equation $ay-bx=1$ or for the equation $ay-bx=-1$. Since this holds if and only if $\mathrm{gcd}(a,b)$ divides $\pm1$, the lemma is proved.
\end{proof}

Let $R$ be a commutative ring. An elementary matrix is an element of $\GL(n,R)$ obtained
by applying elementary transformations to the identity matrix $I_n.$ There
exist three different kinds of elementary matrices, corresponding respectively
to three different types of elementary row transformations (or equivalently,
column transformations):
\begin{itemize}
	\item transpositions $P_{ij},$ with $i \neq j,$ obtained from $I_n$ by exchanging row $i$ and
	row $j;$
	\item  dilations $D_i(u),$ obtained from $I_n$ by multiplying row $i$ by the unit $u \in
	U(R);$
	\item transvections $T_{ij}(r),$ with $i \neq j$ and $r \in R,$ obtained from $I_n$ by adding to
	row $i,$ $r$ times row $j.$ 
\end{itemize}

\begin{lemma}\label{lem}
	Assume  $\begin{pmatrix}a_1&a_2\\b_1&b_2\end{pmatrix}, \begin{pmatrix}c_1&c_2\\d_1&d_2\end{pmatrix} \in \GL(2,\mathbb Z).$ If there exists an elementary matrix $X$ in $\GL(2,\mathbb Z)$ such that
	$$\begin{pmatrix}c_1&c_2\\d_1&d_2\end{pmatrix}=X \begin{pmatrix}a_1&a_2\\b_1&b_2\end{pmatrix}$$ then
	$\alpha:=(a_1,a_2),$ $\beta:=(b_1,b_2),$ $\gamma:=(c_1,c_2)$ and $\delta:=(d_1,d_2)$ belong to the same connected component of $\Gamma(\z \times \z).$
\end{lemma}

\begin{proof}
	The statement is clearly true if $X$ is either a transposition or a dilation. If $X=\begin{pmatrix}1&t\\0&1\end{pmatrix},$ then $\gamma=\alpha+t\beta,$  $\delta=\beta$ and $(\alpha,\beta,\alpha+t\beta)$ is a path in $\Gamma(\z \times \z).$
	Similarly, if $X=\begin{pmatrix}1&0\\t&1\end{pmatrix},$ then $\gamma=\alpha,$  $\delta=t\alpha+\beta$ and $(\beta,\alpha,t\alpha+\beta)$ is a path in $\Gamma(\z \times \z).$
\end{proof}


\begin{thm}Let $G=\z\times \z.$ Then the graph $\Gamma^*(G)$ is connected.
\end{thm}
\begin{proof}
	Assume that $\alpha=(x_1,x_2)$ is a non-isolated vertex of $\Gamma(G).$ There exists $\beta=(y_1,y_2)$ such that $G=\langle \alpha,\beta\rangle.$ This means that $A=\begin{pmatrix}x_1&x_2\\y_1&y_2\end{pmatrix}\in \GL(2,\z),$ so by \cite[Theorem 14.2]{om} there exist $t$ elementary matrices $J_1,\dots,J_t$ such that $A=J_1\cdots J_t.$ By iterated applications of Lemma \ref{lem}, for every $r\in \{1,\dots,t\},$ the rows of the matrix $J_r\cdots J_t$ belong to the connected component of $\Gamma(G)$ containing $(1,0)$ and $(0,1).$ In particular $(1,0)$ and $(x_1,x_2)$ are in the same connected component.
\end{proof}

For every non-isolated vertex $(x,y)$ of $\Gamma(\z\times \z)$ denote by $\mathcal N(x,y)$ the neighbourhood  of $(x,y)$ in $\Gamma(\z\times \z).$ Fix $(x_0,y_0)\in \mathcal N(x,y).$ The matrices in $\GL(2,\mathbb Z)$ having $(x,y)$ as second row are precisely those of the form
$$\begin{pmatrix}kx+x_0&ky+y_0\\x&y\end{pmatrix},\quad \begin{pmatrix}kx-x_0&ky-y_0\\x&y\end{pmatrix} \text{ with $k\in \mathbb Z$}$$
so $\mathcal N(x,y)=\{(kx+x_0,ky+y_0),(kx-x_0,ky-y_0) \mid k\in \mathbb{Z}\}.$
In particular 
$$
 \mathcal  N(0,1)=\{(\pm1,k) \mid k\in \z\}.
$$

\begin{lemma}\label{distance_d}
Let  $(a,b)$ be a vertex in  $\Gamma(\z\times \z)$  at distance  $d$ from $(0,1)$ and let $(x,y)$ be a vertex adjacent to $(a,b)$ and having distance $d-1$ from $(0,1).$ Then there exist  $q_1,\ldots,q_d\in\mathbb Z$ and $\eta_1,\eta_2\in \{0,1\}$ such that 
\begin{equation}\label{distant_t} 
\begin{pmatrix}x&y\\a&b\end{pmatrix}=\begin{pmatrix}0&1\\1&q_d\end{pmatrix}\ldots\begin{pmatrix}0&1\\1&q_2\end{pmatrix}\begin{pmatrix}0&1\\1&q_1\end{pmatrix}(-1)^{\eta_1} E^{\eta_2}
\end{equation}
being $E= 
\begin{pmatrix}-1&0\\0&1\end{pmatrix}
$.
\end{lemma}
\begin{proof}
For any integer $\alpha$, let us set the following matrices:
\begin{equation}
T_\alpha=\begin{pmatrix}0&1\\1&\alpha\end{pmatrix}, \quad T^{*}_\alpha=\begin{pmatrix}0&1\\-1&\alpha\end{pmatrix}.
\end{equation}
We have 
$T_\alpha E= T^{*}_\alpha.$

We argue by induction on $d$. If the vertex $(a,b)$ is at distance   $1$ from $(0,1)$, then $(a,b)$ is  of type $(\pm 1,\alpha)$, where $\alpha$ is an  arbitrary integer. In particular we have
$$\begin{pmatrix}0&1\\1&\alpha\end{pmatrix}= T_\alpha I_2, \quad\begin{pmatrix}0&1\\-1&\alpha\end{pmatrix}= T_\alpha E$$
as desired.

Assume that the statement is true for any vertex at distance  $d-1$  from $(0,1).$ We can write
 $$\begin{pmatrix}x&y\\a&b\end{pmatrix}=T\begin{pmatrix}\tilde x&\tilde y\\x&y\end{pmatrix},$$
 where $(\tilde x,\tilde y)$ is a vertex adjacent to $(x,y)$ at distance  $d-2$ from $(0,1)$ and either $T=T_\alpha$ or $T=T_\alpha^*$ for some $\alpha\in \mathbb Z.$
 Thus, by inductive hypothesis,  there exist   $q_1,\ldots,q_{d-1}\in \mathbb Z$ and $\tilde \eta_1,\tilde \eta_2\in \{0,1\}$ such that  
  $$\begin{pmatrix}x&y\\a&b\end{pmatrix}=TT_{q_{d-1}}\cdots T_{q_1}A$$
  with $A=(-1)^{\tilde \eta_1} E^{\tilde \eta_2}.$
 If $T=T_\alpha$, we are done.
   If $T=T^{*}_\alpha$, then 
  $$T^{*}_\alpha T_{q_{d-1}}\cdots T_{q_1}A=T^{*}_\alpha E(ET_{q_{d-1}}E)\cdots(ET_{q_1}E)EA=T_\alpha T_{-q_{d-1}}\cdots T_{-q_1}(-1)^{d-1}EA.$$
This completes the proof.  
 \end{proof}

Let $R$ be an integral domain, and pick two non-zero elements $a, b \in R.$ Following \cite[Section 14]{om}, we say that
$(a, b)$ belongs to an Euclidean chain  of length $n$ if there is a sequence of divisions
$$r_i = q_{i+1}r_{i+1} + r_{i+2},\ r_i, q_i \in R, \ 0 \leq i \leq n-1,$$
such that $a = r_{0}, b = r_{1}, r_{n} \neq  0$ and $r_{n+1} = 0.$

\begin{lemma}\label{distance} 
If	the vertex $(a,b)$ of $\Gamma^*(\z\times \z)$ has distance  $d$ from $(0,1)$, then $(b, a)$ belongs to an Euclidean chain  of length at most $d.$
\end{lemma}

\begin{proof} 
Let $(a,b)$ be a non-isolated vertex of $\Gamma^*(G)$ at distance   $d$ from $(0,1).$  Lemma \ref{distance_d} tells us that  there exist integers $q_1,\ldots, q_d$ and $\eta_1,\eta_2\in \{0,1\}$ such that 
\begin{equation}\label{decomposition1} 
\begin{pmatrix}*&*\\a&b\end{pmatrix}=\begin{pmatrix}0&1\\1&q_d\end{pmatrix}\ldots\begin{pmatrix}0&1\\1&q_2\end{pmatrix}\begin{pmatrix}0&1\\1&q_1\end{pmatrix}(-1)^{\eta_1}E^{\eta_2}.
\end{equation}
Observe that for any integers $\alpha,\beta$ the following holds:
\[
\begin{pmatrix}0&1\\1&\alpha\end{pmatrix}\begin{pmatrix}0&1\\1&\beta\end{pmatrix}=\begin{pmatrix}1&0\\\alpha&1\end{pmatrix}\begin{pmatrix}1&\beta \\0&1\end{pmatrix}.
\]
Let $t$ be the largest even integer with $t\leq d.$ In view of the previous observation we can rewrite the decomposition in (\ref{decomposition1}) as follows:
\begin{equation}\label{case1}
\begin{pmatrix}*&*\\a&b\end{pmatrix}E^{\eta_2}\begin{pmatrix}1&-q_1\\0&1\end{pmatrix}\begin{pmatrix}1&0\\-q_2&1\end{pmatrix}\ldots\begin{pmatrix}1&-q_{t-1}\\ 0&1\end{pmatrix}\begin{pmatrix}1&0\\-q_t&1\end{pmatrix}=(-1)^{\eta_1}A,
\end{equation}
with 
$$A=\begin{pmatrix}1&0\\0&1\end{pmatrix} \text { if $t$ is even},\quad A=\begin{pmatrix}0&1\\1&q_d\end{pmatrix} \text{ otherwise}.$$

Define
$$r_0=b, \quad r_{1}=(-1)^{\eta_2} a \quad  { \text { and }} \quad r_{i+2}=r_i-q_{i+1}r_{i+1}\quad {\text {for $0\leq i\leq t-1$}}.$$
It follows from (\ref{case1}) that $r_{t+1}=0$ if $d$ is even
and $r_{t+1}=(-1)^{\eta_1}$ if $d$ is odd. In the latter case $0=r_t-qr_{t+1},$
with $q=(-1)^{\eta_1}r_t$,
so in both  cases $(b,(-1)^{\eta_2} a)$ (and consequently $(b,a)$) belongs to an Euclidean chain  of length at most $d$.
\end{proof}

\begin{cor}\label{coro1}
The diameter of $\Gamma^*(\z\times \z)$ is not finite.
\end{cor}
\begin{proof}
	In \cite{laz} Lazard proved that the fastest method to obtain the g.c.d.\! of two integers by means of a succession of generalized  divisions consists in taking, at each step, the division $a=bq+r$ with $|r|\leq |b|/2.$
	In particular consider the series $F_n$ of the Fibonacci numbers. The shortest Euclidean chain for the pair
	$(F_{2n+1},F_{2n})$ is 
	$$\begin{aligned}F_{2n+1}&=2F_{2n}-F_{2n-2}\\F_{2n}&=3F_{2n-2}-F_{2n-4}\\F_{2n-2}&=3F_{2n-4}-F_{2n-6}\\\dots&\dots\dots\dots\dots\dots\dots\\F_4&=3F_2-F_0.\end{aligned}$$  If follows from Lemma \ref{distance} that the distance in the graph $\Gamma^*(\z\times \z)$ of the vertex $(F_{2n},F_{2n+1})$ from $(0,1)$ is at least $n+1,$ so the diameter of $\Gamma^*(\z\times \z)$ is not finite.
\end{proof}

\section{The free group of rank 2}
In this section we will denote by $F$ the free group of rank 2. Recall (see for example  \cite[Theorem 3.2]{cgt}) that the automorphism group of the free group with ordered basis $x_1,x_2$ is generated by the following elementary Nielsen transformations:
\begin{enumerate}
\item switch $x_1$ and $x_2;$
\item replace $x_1$ with $x_1^{-1};$
\item replace $x_1$ with $x_1\cdot x_2.$
\end{enumerate}

\begin{lemma}\label{lemsim}
Let $F=\langle a_1,a_2\rangle
=\langle b_1,b_2\rangle.$ If there exists an elementary Nielsen transformation $\gamma$ such that $a_1^\gamma=b_1$
and $a_2^\gamma=b_2,$ then $a_1$ and $b_1$ (and consequently  $a_2$ and $b_2$) belong to the same connected component of $\Gamma(F).$
\end{lemma}

\begin{proof}
The statement is clearly true if $\gamma$ is of kind (1) or (2).  In the third case $b_1=a_1\cdot a_2,$ so $\langle b_1, a_1\rangle=\langle a_1,a_2\rangle=F.$
\end{proof}

\begin{thm}\label{confree}The graph $\Gamma^*(F)$  is connected.
\end{thm}
\begin{proof}
Fix $x_1,x_2$ such that $F=\langle x_1,x_2\rangle$ and let $a$ be a non-isolated vertex of $\Gamma(F).$ There exists $b$ such that $F=\langle a, b\rangle$ and $\gamma\in \aut(F)$ with $x_1^\gamma=a$ and $x_2^\gamma=b.$
Thus there exist $t$ elementary Nielsen transformations $\gamma_1,\dots,\gamma_t$ such that $\gamma=\gamma_1\cdots \gamma_t.$ By Lemma \ref{lemsim}, $x_1,x_1^{gamma_1}$  and, for every $r\in \{2,\dots,t-1\},$ the two elements $x_1^{\gamma_1\cdots\gamma_{r-1}}$ and $x_1^{\gamma_1\cdots\gamma_{r}}$ belong to the same connected component of $\Gamma(F).$
This implies that  $x_1$ and $x_1^\gamma=a$
are in the same connected component.
\end{proof}

\begin{lemma}\label{rialzo}
Let $f\in F$. If $fF^\prime$ is a non-isolated vertex of $\Gamma(F/F^\prime),$ then there exists a non-isolated vertex $f^*$ of $\Gamma(F)$ such that
 $f^*F^\prime=fF^\prime.$
\end{lemma}
\begin{proof}
Let $\overline F=F/F^\prime$ and for every $x\in F$ let $\overline x=xF^\prime.$ Since $F^\prime$ is a characteristic subgroup of $F$,  there exists a natural epimorphism from $\aut(F)\to \aut(\overline F)\cong\GL(2,\z)$ 
sending any automorphism $\gamma$ of $F$ to the induced automorphism $\overline \gamma$ of $\overline F.$ Notice that $\overline{x^\gamma}=\overline{x}^{\overline {\gamma}}$ for every $x\in F.$ Fix a generating pair $(a,b)$ of $F.$ Since $\overline a$ and $\overline f$ are non-isolated vertices of  $\overline F\cong\z\times \z,$ there exists $\alpha\in \aut(\overline F)$ such that
$\overline f=\overline a^\alpha.$ Let $\gamma\in \aut(F)$ such that $\overline \gamma=\alpha$ and let $f^*=a^\g.$
Since $\langle f^*,b^\gamma\rangle=F,$ the vertex $f^*$ is non-isolated in $\Gamma(F).$ Moreover $\overline{f^*}=\overline{a^\g}=\overline a^{\overline \g}=\overline a^\alpha=\overline f.
$
\end{proof}

\begin{cor}\label{fine}
	The diameter of $\Gamma^*(F)$ is not finite.
\end{cor}
\begin{proof}
Let $d$ be any positive integer. By Corollary \ref{coro1}, there exist two non-isolated vertices $\alpha$ and $\beta$ of $\Gamma(F/F^\prime)$ such that the distance between $\alpha$ and $\beta$ in $\Gamma(F/F^\prime)$ is at least $d$. By Lemma \ref{rialzo}, there exist two non-isolated vertex $a$ and $b$ of $\Gamma(F)$, such that $\alpha=aF^\prime$ and $\beta=bF^\prime.$ We claim that also the distance  between $a$ in $b$ in $\Gamma(F)$ is at least $d.$ Indeed assume that $x_0=a,x_1,\dots,x_{n-1},x_n=b$ is a path in $\Gamma(F)$ joining
$a$ and $b.$ Then $\alpha=x_0F^\prime,x_1F^\prime,\dots,x_{n-1}F^\prime,x_nF^\prime=\beta$ is a path in $\Gamma(F/F^\prime)$ joining
$\alpha$ and $\beta,$ and necessarily we have $n\geq d.$
\end{proof}

\end{document}